\documentclass[11pt]{amsart}
\usepackage[utf8]{inputenc}
\usepackage{amsmath,amsthm,amssymb,amsfonts, mathrsfs,} 
\newtheorem{theorem}{Theorem}[section]

\newtheorem{lemma}[theorem]{Lemma}
\theoremstyle{definition}

\usepackage[dvipsnames]{xcolor}
\usepackage{bbold}

\begin{document}

\title{A remark on the Strichartz Inequality in one dimension}
\author{ Ryan Frier and Shuanglin Shao }

\begin{abstract}
In this paper, we study the extremal problem for the Strichartz inequality for the Schr\"{o}dinger equation on $\mathbb{R}^2$. We show that the solutions to the associated Euler-Lagrange equation are exponentially decaying in the Fourier space and thus can be extended to be complex analytic. Consequently we provide a new proof to the characterization of the extremal functions: the only extremals are Gaussian functions, which was investigated previously by Foschi \cite{Foschi} and Hundertmark-Zharnitsky \cite{Hund}.  
\end{abstract}
\date{\today}
\maketitle

\section{Introduction}

To begin, we note that the Strichartz inequality for an arbitrary dimension $d$ is

\begin{equation*}
\begin{split}
\|e^{it\Delta}f\|_{2+4/d}\leq C_d\|f\|_2,
\end{split}
\end{equation*}

where $\|\ f\|_p=\left(\int_{\mathbb{R}^d} |f(x)|^p\ dx \right)^{1/p}$, and $e^{it\Delta}f=\frac{1}{(2\pi)^d} \int_{\mathbb{R}^d} e^{ix\cdot \xi+i|\xi |^2t}\widehat{f}(\xi)d\xi$, see e.g., \cite{Keel-Tao:1998:endpoint-strichartz,Tao:2006-CBMS-book}. Strichartz's Inequality has long been studied.  The original proof of Strichartz inequality is due to Robert Strichartz in \cite{Strichartz} in 1977.

Define the Fourier transform as $\widehat{f}(\xi)=\int_{\mathbb{R}^d}e^{-ix \cdot \xi}f(x)dx$ and the space-time Fourier transform as $\tilde{F}(\tau, \xi)=\int_{\mathbb{R}\times \mathbb{R}^n}e^{i(\tau t+ x \cdot \xi)} F(t, x) \ dt \ dx$.  Note that in the case of $d=1$ we have 

\begin{equation}\label{eq-1}
\begin{split}
\left\|e^{it\Delta}f\right\|_6 &\leq C_1 \|f\|_2; \\
e^{it\Delta}f&=\frac{1}{2\pi} \int_\mathbb{R} e^{ix\xi+i\xi^2t}\widehat{f}(\xi)d\xi.
\end{split}
\end{equation}

Define
\begin{equation}
\begin{split}
C_1&=\sup \left\{\frac{\left\|e^{it\Delta} f \right\|_6}{\|f\|_2}: f \in L^2, f \neq 0 \right\}.
\end{split}
\end{equation}

We say that $f$ is an \textit{extremizer} or a \textit{maximzer} of the Strichartz inequality if $f\neq 0$ and $\|e^{it\Delta}f\|_6=C_1\|f\|_2$.   The extremal problem for the Strichartz inequality \eqref{eq-1} is (a) Whether there exists an extremzier for \eqref{eq-1}? (b) If it exists, what are the characterizations of extremizers, e.g., continuity and differentiability? What is the explicit formulation of extremizers? are they unique up to the symmetries of the inequality? In this note, we are mainly concerned with question (b).

By D. Foschi's 2007 paper \cite{Foschi} we know that the maximizers are Gaussian functions of the form $f(x)=e^{Ax^2+Bx+C}$, where $A, B, C \in \mathbb{C}$, and $\Re\{A\}<0$ up to the symmetries of the Strichartz inequality.  In particular, according to Foschi, $f(x)=e^{-|x|^2}$ is a maximizer in dimension $1$.  Thus $f$ must satisfy $(1)$, and obtains equality with $C_1$. 

Hundertmark and Zharnitsky in \cite{Hund} showed a new representation using an orthogonal projection operator for dimension $1$ and $2$. The representation that was found is

\begin{equation*}
    \begin{split}
        \int_\mathbb{R} \int_{\mathbb{R}} \left|e^{it\Delta}f(x) \right|^6 \ dx dt&= \frac{1}{2\sqrt{3}} \langle f \otimes f \otimes f, P_1(f \otimes f \otimes f) \rangle_{L^2(\mathbb{R}^3)}, \\
        \int_\mathbb{R} \int_{\mathbb{R}^2} \left| e^{it\Delta}f(x) \right|^4 \ dxdt&= \frac{1}{4} \langle f \otimes f, P_2(f \otimes f) \rangle_{L^2(\mathbb{R}^4)}
        \end{split}
\end{equation*}

for dimensions $d=1$ and $d=2$, respectively, where $P_1, P_2$ are certain projection operators. Using this, they were able to obtain the same results.  In \cite{Kunze} Kunze showed that such a maximizer exists in dimension $1$.  In \cite{Shao2}, the second author showed the existence of a maximizer in all dimensions for the Strichartz inequalities for the Schrödinger equation. Likewise, in \cite{Silva}, Brocchi, Silva, and Quilodrán investigated sharp Strichartz inequalities for fractional and higher order Schrödinger equations.  There they discussed the rapid $L^2$ decay of extremizers, which we will also discuss and use it to establishing a characterization of extremizers.

We will take inspiration from \cite{Shao} to show a different method of proving that extremizers are Gaussians. More precisely, in this note, we are interested in the problem of how to characterize extremals for \eqref{eq-1} via the study of the associated Euler-Lagrange equation. We show that the solutions of this generalized Euler-Lagrange equation enjoy a fast decay in the Fourier space and thus can be extended to be complex analytic, see Theorem \ref{thm-complex-analyticity}. Then as an easy consequence, we give an alternative proof that all extremal functions to \eqref{eq-1} are Gaussians based on solving a functional equation of extremizers derived in Foschi \cite{Foschi}, see \eqref{eq-multiplicative} and Theorem \ref{thm-quadratic}. The functional equality \eqref{eq-multiplicative} is a key ingredient in Foschi's proof in \cite{Foschi}. To prove $f$ in \eqref{eq-multiplicative} to be a Gaussian function, local integrability of $f$ is assumed in \cite{Foschi}, which is further reduced to measurable functions in Charalambides \cite{Charalambides:2012:Cauchy-Pexider}.

Let $f$ be an extremal function to \eqref{eq-1} with the constant $C_1$. Then $f$ satisfies the following generalized Euler-Lagrange equation,
\begin{equation}\label{eq-strichartzequ}
\omega \langle g, f \rangle = \mathcal{Q}(g,f,f,f,f,f), \text{for all }g \in L^2,
\end{equation}
where $\omega= \mathcal{Q}(f,f,f,f,f,f)/\|f\|_{L^2}^2 >0$ and $\mathcal{Q}(f_1,f_2,f_3,f_4,f_5,f_6)$ is the integral
\begin{equation}\label{eq-21} 
\begin{split}
& \int_{\mathbb{R}^6} \overline{\widehat{f_1}}(\xi_1) \overline{\widehat{f_2}}(\xi_2)\overline{\widehat{f_3}}(\xi_3)  \widehat{f_4}(\xi_4)  \widehat{f_5}(\xi_5)\widehat{f_6}(\xi_6) \delta (\xi_1 + \xi_2 +\xi_3-\xi_4-\xi_5-\xi_6) \\
&\qquad\qquad \qquad \times \delta(\xi_1^2 + \xi_2^2 +\xi_3^2- \xi_4^2-\xi_5^2-\xi_6^2) d\xi_1 d\xi_2 d\xi_3 d\xi_4 d\xi_5 d\xi_6,
\end{split}
\end{equation}
for $f_i \in L^2(\mathbb{R})$, $1\le i\le 6$, $\delta (\xi) = (2\pi)^{-d} \int_{\mathbb{R}^d} e^{i \xi \cdot x}dx$ in the distribution sense, $d=1,2$. The proof of \eqref{eq-strichartzequ} is standard; see e.g.  \cite[p. 489]{Evans:2010:PDE-book} or \cite[Section 2]{Hundertmark-Lee:2012:nonlocal-variational-problems-nonlinear-optics} for similar derivations of Euler-Lagrange equations. 

\begin{theorem}\label{thm-complex-analyticity}
If $f$ solves the generalized Euler-Lagrange equation \eqref{eq-strichartzequ} for some $\omega>0$, then there exists $\mu>0$ such that 
$$ e^{\mu |\xi|^2} \widehat{f} \in L^2(\mathbb{R}). $$
Furthermore $f$ can be extended to be complex analytic on $\mathbb{C}$.  
\end{theorem}

To prove this theorem, we follow the argument in \cite{Hundertmark-Shao:Analyticity-of-extremals-Airy-Strichartz}. Similar reasoning has appeared previously in \cite{Erdogan-Hundertmark-Lee:2011:exponential-decay-of-dispersion-management-soliton,Hundertmark-Lee:2009:decay-smoothness-for-solutions-dispersion-managed-NLS}. It relies on a multilinear weighted Strichartz estimate and a continuity argument. See Lemma \ref{le-multilinear} and Lemma \ref{le-bootstrap}, respectively. 

Next we prove that the extremals to \eqref{eq-1} are Gaussian functions. We start with the study of the functional equation derived in \cite{Foschi}. In \cite{Foschi}, the functional equation reads
\begin{equation}\label{eq-multiplicative}
f(x)f(y)f(z) = f(a)f(b)f(c),
\end{equation}
for any $x,y,z,a,b,c \in \mathbb{R}$ such that 
\begin{equation}\label{eq-4}
x+y+z =a+b+c, \quad x^2 + y^2+y^2 = a^2 + b^2+c^2,
\end{equation}
In \cite{Foschi}, it is proven that $f\in L^2$ satisfies \eqref{eq-multiplicative} if and only if $f$ is an extremal function to \eqref{eq-1}. Basically, this comes from two aspects. One is that in the Foschi's proof of the sharp Strichartz inequality only the Cauchy-Schwarz inequality is used at one place besides equality. So the equality in the Strichartz inequality \eqref{eq-1}, or equivalently the equality in Cauchy-Schwarz,  yields the same functional equation as \eqref{eq-multiplicative} where $f$ is replaced by $\hat{f}$. The other one is that the Strichartz norm for the Schr\"odinger equation enjoys an identity that \begin{equation}\label{eq-symmetry}
\|e^{it\Delta}f \|_{L^6(\mathbb{R}^2)} =C  \|e^{it\Delta}f^\vee \|_{L^6(\mathbb{R}^2)}
\end{equation}for some $C>0$.

In \cite{Foschi}, Foschi is able to show that all the solutions to \eqref{eq-multiplicative} are Gaussians under the assumption that $f$ is a locally integrable function. In \cite{Shao}, Jiang and the second author studied the two dimensional case of \eqref{eq-multiplicative} and proved that the solutions are Gaussian functions. These can be viewed as investigations of the Cauchy functional equations \eqref{eq-multiplicative} for functions supported on the paraboloids. To characterize the extremals for the Tomas-Stein inequality for the sphere in $\mathbb{R}^3$, in \cite{Christ-Shao:extremal-for-sphere-restriction-II-characterizations},  Christ and the second author study the functional equation of similar type for functions supported on the sphere and prove that they are exponentially affine functions.  In \cite{Charalambides:2012:Cauchy-Pexider}, Charalambides generalizes the analysis in \cite{Christ-Shao:extremal-for-sphere-restriction-II-characterizations} to some general hyper-surfaces in $\mathbb{R}^n$ that include the sphere, paraboloids and cones as special examples and proves that the solutions are exponentially affine functions. In \cite{Charalambides:2012:Cauchy-Pexider, Christ-Shao:extremal-for-sphere-restriction-II-characterizations}, the functions are assumed to be measurable functions.  

By the analyticity established in Theorem \ref{thm-complex-analyticity}, Equations \eqref{eq-multiplicative} and \eqref{eq-4} have the following easy consequence, which recovers the result in \cite{Foschi, Hund}. 

\begin{theorem}\label{thm-quadratic}
Suppose that $f$ is an extremal function to \eqref{eq-1}. Then 
\begin{equation}\label{eq-quadratic}
f(x) = e^{A |x|^2+ B \cdot x + C},
\end{equation} 
where $A, C \in \mathbb{C}, B \in \mathbb{C}$ and $\Re(A)<0$. 
\end{theorem}

\section{Developing the Extremizer}
We want to show that if $f$ solves the generalized Euler-Lagrange equation \eqref{eq-strichartzequ}, then there exists some $\mu>0$ such that 

\begin{equation*}
\begin{split}
e^{\mu|\xi|^2}\widehat{f} \in L^2.
\end{split}
\end{equation*}

Furthermore, we can extend $f$ to be entire.  To begin, we note that by Foshi's paper \cite{Foschi}, we have for a maximizer $f$, $f(x)f(y)f(z)=F(x^2+y^2+z^2, x+y+z)$.  Thus for any $(x, y, z), (a, b, c) \in \mathbb{R}^3$ such that

\begin{equation}\label{eq-a1}
\begin{split}
x+y+z&=a+b+c
\end{split}
\end{equation}
and 
\begin{equation}\label{eq-a2}
\begin{split}
x^2+y^2+z^2&=a^2+b^2+c^2,
\end{split}
\end{equation}
then 
\begin{equation*}
\begin{split}
f(x)f(y)f(z)&=F(x^2+y^2+z^2, x+y+z) \\
&=F(a^2+b^2+c^2, a+b+c) \\
&=f(a)f(b)f(c).
\end{split}
\end{equation*}
That is, 
\begin{equation}\label{eq-a3}
\begin{split}
f(x)f(y)f(z)&=f(a)f(b)f(c).
\end{split}
\end{equation}

Let us assume that $f$ has an entire extension. Then $f$ restricted to $\mathbb{R}$ is real analytic. By \cite[Lemma 7.9]{Foschi}, such nontrival $f\in L^2$ is also nonzero. We prove the following theorem.

\begin{theorem}
If $f$ is a maximizer for $(1)$, then $f(x)=e^{Ax^2+Bx+C}$, where $A, B, C \in \mathbb{C}$.  
\end{theorem}

\begin{proof}
Consider $\varphi(x)=\log(f(x))$. We know from \cite[Lemma 7.9]{Foschi} that $f$ is nowhere $0$, so $\varphi$ is well defined.  Since $f$ is analytic, then so is $\varphi$. Hence by the power series expansion we have
\begin{equation*}
\varphi(x)=\varphi(0)+\varphi'(0)x+\frac{\varphi''(0)}{2}x^2+\sum_{k\geq 3} \frac{\varphi^{(k)}(0)}{k!}x^k.
\end{equation*}
Hence it is true for $a, b, c, d, e, g$ such that $(a, b, c), (d, e, g)$ satisfy equations \eqref{eq-a1} and \eqref{eq-a2}.  That is,
\begin{equation*}
\begin{split}
\varphi(a)&=\varphi(0)+\varphi'(0)a+\frac{\varphi''(0)}{2}a^2+\sum_{k\geq 3} \frac{\varphi^{(k)}(0)}{k!}a^k, \\
\varphi(b)&=\varphi(0)+\varphi'(0)b+\frac{\varphi''(0)}{2}b^2+\sum_{k\geq 3} \frac{\varphi^{(k)}(0)}{k!}b^k, \\
\varphi(c)&=\varphi(0)+\varphi'(0)c+\frac{\varphi''(0)}{2}c^2+\sum_{k\geq 3} \frac{\varphi^{(k)}(0)}{k!}c^k, \\
\varphi(d)&=\varphi(0)+\varphi'(0)d+\frac{\varphi''(0)}{2}d^2+\sum_{k\geq 3} \frac{\varphi^{(k)}(0)}{k!}d^k, \\
\varphi(e)&=\varphi(0)+\varphi'(0)e+\frac{\varphi''(0)}{2}e^2+\sum_{k\geq 3} \frac{\varphi^{(k)}(0)}{k!}e^k, \\
\varphi(g)&=\varphi(0)+\varphi'(0)g+\frac{\varphi''(0)}{2}g^2+\sum_{k\geq 3} \frac{\varphi^{(k)}(0)}{k!}g^k.
\end{split}
\end{equation*}
By equation \eqref{eq-a3} we know that $\varphi(a)+\varphi(b)+\varphi(c)=\varphi(d)+\varphi(e)+\varphi(g)$.  Thus by using the power series expansions and equations \eqref{eq-a1}, \eqref{eq-a2} and \eqref{eq-a3} we have that
\begin{equation*}
\begin{split}
0&=\varphi(a)+\varphi(b)+\varphi(c)-\varphi(d)-\varphi(e)-\varphi(g) \\
&=\varphi(0)+\varphi'(0)a+\frac{\varphi''(0)}{2}a^2+\sum_{k\geq 3} \frac{\varphi^{(k)}(0)}{k!}a^k \\
&+\varphi(0)+\varphi'(0)b+\frac{\varphi''(0)}{2}b^2+\sum_{k\geq 3} \frac{\varphi^{(k)}(0)}{k!}b^k \\
&+\varphi(0)+\varphi'(0)c+\frac{\varphi''(0)}{2}c^2+\sum_{k\geq 3} \frac{\varphi^{(k)}(0)}{k!}c^k \\
&-\varphi(0)-\varphi'(0)d-\frac{\varphi''(0)}{2}d^2-\sum_{k\geq 3} \frac{\varphi^{(k)}(0)}{k!}d^k \\
&-\varphi(0)-\varphi'(0)e-\frac{\varphi''(0)}{2}e^2-\sum_{k\geq 3} \frac{\varphi^{(k)}(0)}{k!}e^k \\
&-\varphi(0)-\varphi'(0)g-\frac{\varphi''(0)}{2}g^2-\sum_{k\geq 3} \frac{\varphi^{(k)}(0)}{k!}g^k \\
&= \varphi'(0)(a+b+c-d-e-g)+\frac{\varphi''(0)}{2}(a^2+b^2+c^2-d^2-e^2-g^2) \\
&\qquad \qquad +\sum_{k\geq 3}\frac{\varphi^{(k)}(0)}{k!}(a^k+b^k+c^k-d^k-e^k-g^k) \\
&=\sum_{k\geq 3}\frac{\varphi^{(k)}(0)}{k!}(a^k+b^k+c^k-d^k-e^k-g^k).
\end{split}
\end{equation*}
That is, 
\begin{equation}
\begin{split}
\sum_{k \geq 3} \frac{\varphi^{(k)}(0)}{k!}(a^k+b^k+c^k-d^k-e^k-g^k)&=0,
\end{split}
\end{equation}
where $(a, b, c), (d, e, g) \in \mathbb{R}^3$ satisfy equations \eqref{eq-a1} and \eqref{eq-a2}. Consider $a=x,b=-x, c=x,g=0$, by solving the equations 
\begin{equation*}
\begin{split}
  d+e    & =x, \\  
  d^2+e^2 &=3x^2,
\end{split}
\end{equation*}
we obtain that $d=\frac {1+\sqrt 5}{2} x, e= \frac {1-\sqrt 5 }{2} x$. Then $$ a^k+b^k+c^k -e^k-f^k-g^k = 2x^k+(-x)^k- \left(\frac {1+\sqrt 5}{2} x \right)^k-\left(\frac {1-\sqrt 5 }{2} x \right)^k.$$
When $k$ is even,
$$-\left( 3-(\frac {1+\sqrt 5}{2})^k-(\frac {1-\sqrt 5}{2})^k \right) \ge (\frac 32)^k -3>0$$
for $k\ge 3$. 
When $k$ is odd, 
$$-\left( 1-(\frac {1+\sqrt 5}{2})^k+ (\frac {\sqrt 5-1}{2})^k \right) \ge (\frac 32)^k -2>0 $$
for $k\ge 3$. 
This shows that $\varphi^k(0)=0$ when $k \geq 3$. Hence $\varphi^{(k)}(0)=0$ for all $k$.  Thus $\varphi(x)=Ax^2+Bx+C$.  Therefore $f(x)=e^{Ax^2+Bx+C}$.
\end{proof}

\section{Establishing the Exponential Decay in Fourier Space}
Consider the integral
\begin{equation*}
\begin{split}
&Q(f_1, f_2, f_3, f_4, f_5, f_6) = \\
&\int_{\mathbb{R}^6}\overline{\widehat{f_1}}(\xi_1)\overline{\widehat{f_2}}(\xi_2) \overline{\widehat{f_3}}(\xi_3)\widehat{f_4}(\xi_4)\widehat{f_5}(\xi_5) \widehat{f_6}(\xi_6)\delta(\xi_1+\xi_2+\xi_3-\xi_4-\xi_5-\xi_6)\times  \\
&\delta(\xi_1^2+\xi_2^2+\xi_3^2-\xi_4^2-\xi_5^2-\xi_6) d \xi_1d\xi_2d\xi_3d\xi_4d\xi_5d\xi_6.
\end{split}
\end{equation*}

Notice that if $f$ is an extremal function to Strichartz estimate, then $f$ must satisfy the generalized Euler-Lagrange equation

\begin{equation}
\begin{split}
\omega \langle g, f \rangle &=Q(g, f, f, f, f, f)
\end{split}
\end{equation}
for all $g \in L^2$, where $\omega=\frac{1}{||f||^2_2}Q(f, f, f, f, f, f)>0$, and $\delta(\xi)=\frac{1}{2\pi}\int_\mathbb{R} e^{i\xi x} dx$ in the distribution sense.

Define 
\begin{equation*}
\begin{split}
\eta&:=(\eta_1, \eta_2, \eta_3, \eta_4, \eta_5, \eta_6) \in \mathbb{R}^6, \\
a(\eta)&:= \eta_1+\eta_2+\eta_3-\eta_4-\eta_5-\eta_6, \\
b(\eta)&:= \eta_1^2+\eta_2^2+\eta_3^2-\eta_4^2-\eta_5^2-\eta_6^2.
\end{split}
\end{equation*}
The choice of $a(\eta)$ and $b(\eta)$ are useful, since when $a(\eta)=0=b(\eta)$, then $\eta$, more specifically, $(\eta_1, \eta_2, \eta_3), (\eta_4, \eta_5, \eta_6) \in \mathbb{R}^3$, satisfy equations \eqref{eq-a1} and \eqref{eq-a2}.  For $\varepsilon\geq 0, \ \mu \geq 0,$ and $\xi \in \mathbb{R}$ define 
\begin{equation*}
\begin{split}
F(\xi)&:= F_{\mu, \varepsilon}(\xi)= \frac{\mu \xi^2}{1+\varepsilon \xi^2}.
\end{split}
\end{equation*}
For $h_i \in L^2(\mathbb{R}), \ 1 \leq i \leq 6$, define the weighted multilinear integral $M_F$ as 
\begin{equation*}
\begin{split}
M_F(h_1, h_2, h_3, h_4, h_5, h_6)&=\int_{\mathbb{R}^6} e^{F(\eta_1)-\sum_{k=2}^6 F(\eta_k)} \Pi_{k=1}^6 |h(\eta_k)|\delta(a(\eta)) \delta(b(\eta)) d \eta.
\end{split}
\end{equation*}

It is easy to see that 
\begin{equation}\label{eq-3}
\begin{split}
M_F(h_1, h_2, h_3, h_4, h_5, h_6) &\leq \int_{\mathbb{R}^6} \Pi_{k=1}^6 |h(\eta_k)|\delta(a(\eta)) \delta(b(\eta)) d \eta.
\end{split}
\end{equation}

Indeed, on the support of $a$ and $b$, 
$$\eta_1^2 \le \sum_{i=2}^6 \eta_i^2.$$  
We also note that $F(\xi)$ is an increasing function, and $F(\xi)\geq 0 $ for all $\xi, \ \mu$, and $\varepsilon$. So equation \eqref{eq-3} can be derived by
\begin{equation*}
\begin{split}
\left|M_F(h_1, h_2, h_3, h_4, h_5, h_6) \right|&= \left|\int_{\mathbb{R}^6} e^{F(\eta_1)-\sum_{k=2}^6 F(\eta_k)} \Pi_{k=1}^6 |h(\eta_k)|\delta(a(\eta)) \delta(b(\eta)) d \eta \right| \\
&\leq \int_{\mathbb{R}^6}\left|e^{F(\eta_1)-\sum_{k=2}^6 F(\eta_k)} \right| \Pi_{k=1}^6 \left|h(\eta_k) \right| \delta(a(\eta)) \delta(b(\eta)) d \eta \\
&\leq \int_{\mathbb{R}^6} \Pi_{k=1}^6 \left|h(\eta_k) \right| \delta(a(\eta)) \delta(b(\eta)) d \eta.
\end{split}
\end{equation*}
We state the following key lemma, which is established by using the Hausdorff-Young inequality. The two dimensional such estimate is due to Bourgain \cite{Bourgain:1998:refined-Strichartz-NLS} that is much harder. 
\begin{lemma}\label{Pre-Bourgain}

\begin{equation}\label{eq-2}
\begin{split}
\left\| e^{it\Delta} h_1 e^{it \Delta}h_2 \right\|_{L^3_{t, x}}&\leq CN^{-1/6} \|h_1\|_{L^2} \|h_2\|_{L^2},
\end{split}
\end{equation}
where $h_1 \in L^2$ is supported on $|\xi|\leq s$ and $h_2\in L^2$ is supported on $|\eta|\geq Ns$, for $N\gg1$ and $s \gg 1$. 
\end{lemma}

Equation \eqref{eq-2} has been established in \cite{Silva}. We provide a proof for completeness. Let $\widehat{f}$ be supported on $|\xi|\leq s$ and $\widehat{g}$ be supported on $|\eta|\geq Ns$ where $N \gg 1$ and $s\gg1$, and note that 
\begin{equation*}
\begin{split}
e^{it\Delta}fe^{it\Delta}g&=\frac{1}{(2\pi)^2} \int_\mathbb{R} \int_\mathbb{R} e^{ix\xi+it\xi^2}\widehat{f}(\xi)e^{ix\eta+it\eta^2}\widehat{g}(\eta)d\eta d \xi \\
&=\frac{1}{(2\pi)^2}\int_\mathbb{R} \int_\mathbb{R}e^{ix(\xi+\eta)+it(\xi^2+\eta^2)}\widehat{f}(\xi)\widehat{g}(\eta)d\eta d \xi.
\end{split}
\end{equation*}
Consider the change of variables $\gamma=\xi+\eta$ and $\tau=\xi^2+\eta^2$.  Let $|J|=\frac{1}{\sqrt{2\tau-\gamma^2}}$ be the corresponding Jacobian.  Thus $2\tau-\gamma^2=(\xi-\eta)^2$.  Let $|\xi|\leq s$ and $|\eta|\geq Ns$, where $N>1$.  If $\eta>0$, then $\eta>\xi$.  If $\eta<0$, then $\eta<\xi$.  In either case, the Jacobian is well defined.  Likewise, by considering $2^kNs \leq |\eta|\leq 2^{k+1}Ns$, we see that $|J|\lesssim (2^kNs)^{-1}$. Let 

\begin{equation*}
\begin{split}
G(\gamma, \tau)&:=\widehat{f} \left( \frac{\gamma+ \sqrt{2\tau-\gamma^2}}{2} \right) \widehat{g} \left(\frac{\gamma- \sqrt{2\tau-\gamma^2}}{2} \right)|J|.
\end{split}
\end{equation*}
Then we have 
\begin{equation*}
\begin{split}
e^{it\Delta}fe^{it\Delta}g&=\frac{1}{(2\pi)^2}\int e^{ix\gamma+it\tau}G(\gamma, \tau) d(\gamma \times \tau) \\
&=\frac{1}{(2\pi)^2}\widetilde{G}(x, t).
\end{split}
\end{equation*}
Thus by the Hausdorff-Young inequality and change of variables, we have 
\begin{align*}
\left\|e^{it\Delta}f e^{it\Delta}g \right\|_{L^3_{x, t}}&=\left\|\tilde{G}(x, t) \right\|_{L^3_{x, t}} \leq \left\|G(\gamma, \tau) \right\|_{L^{3/2}_{\gamma, \tau}} =\left( \int |G(\gamma, \tau)|^{3/2} d (\gamma \times \tau) \right)^{\frac{2}{3}} \\
&= \left( \int \left|\widehat{f}\left( \frac{\gamma \pm \sqrt{2\tau-\gamma^2}}{2} \right) \right|^{\frac{3}{2}} \left| \widehat{g} \left( \frac{\gamma \mp \sqrt{2\tau - \gamma^2}}{2} \right) \right|^{\frac{3}{2}} \left|J\right|^{\frac{3}{2}} d(\gamma \times \tau) \right)^{\frac{2}{3}} \\
&= \left( \int \left| \widehat{f}( \xi) \right|^{\frac{3}{2}} \left| \widehat{g}(\eta) \right|^{\frac{3}{2}}|J|^{\frac{3}{2}} |J|^{-1} d(\xi \times \eta) \right)^{\frac{2}{3}}.
\end{align*}
The above continues to equal
\begin{align*}
&\left( \sum_{k=0}^\infty \int_{|\xi|\leq s, 2^{k}Ns \leq |\eta| \leq 2^{k+1}Ns} \left| \widehat{f}( \xi) \right|^{\frac{3}{2}} \left| \widehat{g}(\eta) \right|^{\frac{3}{2}} |J|^{\frac{1}{2}} d(\xi \times \eta) \right)^{\frac{2}{3}} \\
& \leq \sum_{k=0}^\infty \left(  \int_{|\xi|\leq s, 2^{k}Ns \leq |\eta| \leq 2^{k+1}Ns} \left| \widehat{f}( \xi) \right|^{\frac{3}{2}} \left| \widehat{g}(\eta) \right|^{\frac{3}{2}} |J|^{\frac{1}{2}} d(\xi \times \eta) \right)^{\frac{2}{3}}\\
& \lesssim\sum_{k=0}^\infty  \left( \int \left| \widehat{f}( \xi) \right|^{\frac{3}{2}} \left| \widehat{g}(\eta) \right|^{\frac{3}{2}} (2^ksN)^{-\frac{1}{2}} d(\xi \times \eta) \right)^{\frac{2}{3}} \\
&=(sN)^{-\frac{1}{3}}\sum_{k=0}^\infty 2^{-\frac{k}{3}}\left(\int_{|\xi | \leq s, 2^kNs \leq |\eta | \leq 2^{k+1}Ns} \left| \widehat{f}( \xi) \right|^{\frac{3}{2}} \left| \widehat{g}(\eta) \right|^{\frac{3}{2}}  d(\xi \times \eta) \right)^{\frac{2}{3}} \\
&=(sN)^{-\frac{1}{3}}\left( \int_{|\xi|\leq s} |\widehat{f}(\xi)|^{\frac{3}{2}} d \xi \right)^{\frac{2}{3}}\sum_{k=0}^\infty2^{-\frac{k}{3}}  \left( \int_{ 2^kNs \leq |\eta|\leq 2^{k+1} Ns}  |\widehat{g}(\eta) |^{\frac{3}{2}} d\eta \right)^{\frac{2}{3}}.
\end{align*}
For $\left(\int_{|\xi|\leq s} |\widehat{f}(\xi)|^{3/2}\right)^{2/3}$, we wish to show that $\int_{|\xi|\leq s} \left| \widehat{f}(\xi) \right|^{3/2} \lesssim s^{1/4} \|f\|_2^{3/2}$.  Consider Hölder's inequality, for $p=4$ and $q=\frac{4}{3}$.  Then 
\begin{equation*}
\begin{split}
\int_{|\xi|\leq s} \left| \widehat{f} (\xi) \right|^{3/2} &= \int_{|\xi| \leq s} 1 \cdot \left| \widehat{f}(\xi) \right|^{3/2} \\
& \leq \left( \int_{|\xi|\leq s} 1^4 \right)^{1/4} \left( \int_{|\xi|\leq s} \left|\ \widehat{f}(\xi) \right|^2 \right)^{3/4} \\
&=(2s)^{1/4} \left( \left(\int_{|\xi|\leq s} \left| \widehat{f}(\xi) \right|^2 \right)^{1/2} \right)^{3/2} \\
&\leq (2s)^{1/4}\|\widehat{f} \|_2^{3/2} = (2s)^{1/4} \|f\|_2^{3/2},
\end{split}
\end{equation*}
where the final step is a consequence of Plancherel's theorem.  Hence $$\left( \int_{|\xi| \leq s} \left| \widehat{f}(\xi) \right|^{3/2} \right)^{2/3} \leq (2s)^{1/6} \|f\|_2.$$

As for $\left(\int_{2^{k}Ns \leq |\eta|\leq 2^{k+1}Ns} |\widehat{g}(\eta)|^{3/2}\right)^{2/3}$, we use a similar technique as we did for $\widehat{f}$.  Specifically,
\begin{equation*}
\begin{split}
\int_{2^kNs \leq |\eta| \leq 2^{k+1}Ns} \left| \widehat{g}(\eta) \right|^{3/2} & \leq \left( \int_{2kNs \leq |\eta| \leq 2^{k+1}Ns} 1^4 \right)^{1/4} \left( \int_{2^kNs \leq |\eta| \leq 2^{k+1}Ns} \left| \widehat{g}(\eta) \right|^2 \right)^{3/4} \\
&=\left(2^kNs\right)^{1/4} \left( \left( \int_{2^kNs \leq |\eta| \leq 2^{k+1}Ns} \left| \widehat{g}(\eta) \right|^2 \right)^{1/2} \right)^{3/2} \\
&\leq \left(2^kNs \right)^{1/4} \|\widehat{g} \|_2^{3/2} \\
&=\left(2^kNs \right)^{1/4} \|g\|_2^{3/2}.
\end{split}
\end{equation*}
Hence $\left(\int_{2^kNs \leq |\eta| \leq 2^{k+1} Ns} \left| \widehat{g}(\eta) \right|^{3/2}\right)^{2/3} \leq (2^kNs)^{1/6} \|g\|_2$.  By pairing this with the above we have

\begin{equation*}
\begin{split}
\left\|e^{it\Delta}f e^{it\Delta}g \right\|_{L^3_{x, t}}& \lesssim(sN)^{-\frac{1}{3}}\left( \int_{|\xi|\leq s} |\widehat{f}(\xi)|^{\frac{3}{2}} d \xi \right)^{\frac{2}{3}} \\
&\qquad \times \sum_{k=0}^\infty2^{-\frac{k}{3}}  \left( \int_{ 2^kNs \leq |\eta|\leq 2^{k+1} Ns}  |\widehat{g}(\eta) |^{\frac{3}{2}} d\eta \right)^{\frac{2}{3}} \\
& \leq (sN)^{-1/3} (2s)^{1/6} \|f\|_2 \sum_{k=0}^\infty 2^{-k/3} (2^kNs)^{1/6} \|g\|_2 \\
&=2^{1/6}N^{-1/6} \|f\|_2 \|g\|_2 \sum_{k=0}^\infty 2^{-k/6} \\
&= CN^{-1/6} \|f\|_ 2 \|g\|_2.
\end{split}
\end{equation*}

If we pair the estimate in Lemma \ref{Pre-Bourgain} with Hölder's inequality and the $L^6 \rightarrow L^2$ Strichartz inequality, we get the following lemma.
\begin{lemma}\label{le-multilinear}
Let $h_k \in L^2(\mathbb{R}), \ 1 \leq k \leq 6$, and $s\gg 1, \ N \gg 1$.  Suppose that the Fourier transform of $h_1$ is supported on $\{\xi: |\xi| \leq s\}$ and the Fourier transform of $h_2$ is supported on $\{|\xi|\geq Ns\}$.  Then 
\begin{equation*}
\begin{split}
M_F(h_1,h_2, h_3, h_4, h_5, h_6)&\leq CN^{-1/6} \Pi_{k=1}^6 \|h_k\|_2.
\end{split}
\end{equation*}
\end{lemma}

Next we focus on establishing Theorem \ref{thm-complex-analyticity}. If it can be shown that $e^{\mu \xi^2}\widehat{f} \in L^2$ for some $\mu>0$, then $e^{\lambda |\xi|^2} \widehat{f} \in L^1$ for some $0<\lambda<\mu$. Then by the Fourier inversion equation we have that $f(z)=\frac{1}{2\pi} \int_\mathbb{R} e^{iz\xi-\lambda |\xi|^2}e^{\lambda |\xi|^2} \widehat{f}(\xi)d\xi$.  Thus 

\begin{equation*}
\begin{split}
\partial_{\overline{z}}f(z)&= \partial_{\overline{z}} \left( \frac{1}{2\pi} \int_\mathbb{R} e^{iz\xi-\lambda |\xi|^2}e^{\lambda |\xi|^2}\widehat{f}(\xi) d\xi \right)\\
&=\int_\mathbb{R} \partial_{\overline{z}}\left( e^{iz\xi-\lambda|\xi|^2}\right)e^{\lambda|\xi|^2} \widehat{f}(\xi) d\xi=0.
\end{split}
\end{equation*}
So $f$ can be extended to complex analytic on $\mathbb{C}$. To prove Theorem \ref{thm-complex-analyticity}, we establish

\begin{lemma}\label{le-bootstrap} Let $f$ solve the generalized Euler-Lagrange equation $(7)$ for $\omega$ as defined just below equation $(7)$, $\|f\|_2=1$, and define $\widehat{f}_>:= \widehat{f} 1_{|\xi|\geq s^2}$ for $s>0$.  Then there exists some $s \gg 1$ such that for $\mu=s^{-4}$, 
\begin{equation}
\begin{split}
\omega \left\|e^{F(\cdot)}\widehat{f}_> \right\|_2& \leq o_1(1) \left\|e^{F(\cdot)}\widehat{f}_> \right\|_2+C\left\|e^{F(\cdot)}\widehat{f}_> \right\|_2^2+C\left\|e^{F(\cdot)}\widehat{f}_> \right\|^3_2\\
&\qquad +C\left\|e^{F(\cdot)}\widehat{f}_> \right\|_2^4+C\left\|e^{F(\cdot)}\widehat{f}_> \right\|_2^5+o_2(1),
\end{split}
\end{equation}
where $\lim_{s \rightarrow \infty} o_i(1)=0$ uniformly for all $\varepsilon>0, \ i=1, 2$, and the constant $C$ is independent of $\varepsilon$ and $s$.
\end{lemma}

\begin{proof}
For this proof we follow the proof of lemma $2.2$ in \cite{Shao}.  Note that $\left\|e^{F(\cdot)} \widehat{f}_> \right\|_2^2=\langle e^{F(\cdot)}\widehat{f}_>, e^{F(\cdot)}\widehat{f}_> \rangle=\langle e^{2F(\cdot)}\widehat{f}_>, \widehat{f} \rangle=\langle e^{2F(\cdot)}f_>, f \rangle$.  So by equation $(7)$ we have that 

\begin{equation*}
\begin{split}
\omega &\left\|e^{F(\cdot)}\widehat{f}_> \right\|_2^2=Q(e^{2F(\cdot )} f_>, f, f, f, f, f) \\
&\quad =\int_\mathbb{R^6} e^{F(\xi_1)-\sum_{k=2}^6 F(\xi_k)}h_>(\xi_1)h(\xi_2)h(\xi_3)h(\xi_4)h(\xi_5)h(\xi_6) \delta(a(\xi))\delta(b(\xi))d\xi,
\end{split}
\end{equation*}
where $\xi=(\xi_1, \xi_2, \xi_3, \xi_4, \xi_5, \xi_6)$, and $h(\xi_i):=e^{F(\xi_i)}\widehat{f}(\xi_i)$ and $h_>(\xi_i):=e^{F(\xi_i)}\widehat{f}_>(\xi_i)$ for $1 \leq i \leq 6$.  Thus $\omega\left\|e^{F(\cdot)}\widehat{f}_> \right\|_2^2\leq M_F(h_>, h, h, h, h, h)$.  Define $h_\ll:=h 1_{|\xi|<s}$ and $h_\sim:=h1_{s\leq |\xi| \leq s^2}$.  We can break $M_F$ up into the intervals $|\xi|< s, \ |\xi|\geq s^2,$ and $s\leq |\xi| < s^2$ so that $h_<= h1_{|\xi|<s^2}$.  Thus
\begin{equation*}
\begin{split}
M_F(h_>, h, h, h, h, h)&= M_F(h_>, h_<, ..., h_<) + \sum_{j_2, ..., j_6} M_F(h_>, h_{j_2}, ..., h_{j_6} ) \\
&:= A+B,
\end{split}
\end{equation*}
where $j_i$ is either $<$ or $>$, and at least one of the subscripts is $>$.  We break up $A$ into $M_F(h_>, h_\ll, h_<, ..., h_<)+M_F(h_>, h_\sim, h_<, ..., h_<):=A_1+A_2$.  By lemma $3.1$ we know that $A_1 \lesssim s^{-1/6} \|h_>\|_2 \|h_\ll \|_2 \|h_<\|_2^4$. Note that 
\begin{equation*}
\begin{split}
\|h_< \|^2_2&= \int \left| e^{F(\xi)}\widehat{f}(\xi) 1_{|\xi| < s^2}\right|^2 = \int e^{2 \frac{\mu\xi^2}{1+\varepsilon \xi^2}} \left| \widehat{f} \right|^2 1_{|\xi|<s^2} \\
&\leq e^{2\mu s^4} \|f\|_2^2 =e^{2\mu s^4}.
\end{split}
\end{equation*}
So  $\|h_<\|_2 \leq e^{\mu s^4}$.  Likewise, $\|h_\ll\|_2 \leq e^{\mu s^2}$.  As for $A_2$, if we similarly define $f_\sim$, we have $\|h_\sim\|_2\leq e^{\mu s^4} \|f_\sim\|_2$.  To see that $\|f_\sim\|_2 \rightarrow 0$ as $s \rightarrow \infty$, recall that 
\begin{equation*}
\begin{split}
\|f\|_2^2&=\sum_{k=1}^\infty \int_{x_k \leq | \xi | \leq x_{k+1}} |f|^2,
\end{split}
\end{equation*}
where $\{x_k\}$ is a sequence such that $x_0=0$, and $x_k$ is strictly increasing.  Since $\|f\|_2=1$, then $\lim_{k \rightarrow \infty} \int_{x_k \leq |\xi| \leq x_{k+1}}|f|^2 =0$. Thus $\|f_\sim \|_2 \rightarrow 0$ as $s \rightarrow \infty$.  So
\begin{equation*}
\begin{split}
A&=A_1+A_2 \\
&\lesssim s^{-1/6} \|h_>\|_2 \|h_\ll \|_2 \|h_<\|_2^4+\|h_>\|_2\|h_\sim \|_2 \|h_<\|_2^4 \\
&\leq s^{-1/6} \|h_>\|_2 e^{\mu s^2}e^{\mu s^4}+ \|h_>\|_2 e^{\mu s^4} \|f_\sim \|_2 e^{\mu s^4} \\
&=e^{2\mu s^4} \|h_>\|_2 \left( s^{-1/6}e^{\mu s^2-\mu s^4}+ \|f_\sim \|_2 \right) \\
&=o_1(1) \|h_>\|_2 =o_1(1) \left\|e^{F(\cdot)}\widehat{f} \right\|_2,
\end{split}
\end{equation*}
where $o_1(1) = e^{2\mu s^4} \left(s^{-1/6}e^{\mu s^2-\mu s^4}+ \|f_\sim \|_2\right)$.  Thus $o_1(1) \rightarrow 0$ as $s \rightarrow \infty$.  

As for $B$, let $B_1:=\sum_{j_2, ..., j_6} M_F(h_>, h_{j_2}, ..., h_{j_6} )$ containing precisely $1$ $h_>\in \{h_{j_2}, ..., h_{j_6}\}$, $B_k:= \sum_{j_2, ..., j_6} M_F(h_>, h_{j_2}, ..., h_{j_6})$ containing precisely $k \ h_> \in \{h_{j_2}, ..., h_{j_6}\}$. For example, 
\begin{equation*}
\begin{split}
B_1&=M_F(h_>, h_>, h_<, h_<, h_<, h_<)+M_F(h_>, h_<, h_>, h_<, h_<, h_<) \\
&+M_F(h_>, h_<, h_<, h_>, h_<, h_<)+M_F(h_>, h_<, h_<, h_<, h_>, h_<) \\
&+M_F(h_>, h_<, h_<, h_<, h_<, h_>) \\
&=CM_F(h_>, h_<, h_>, h_<, h_<, h_<) \\
&=CM_F(h_>, h_\ll, h_>, h_<, h_<, h_<)+CM_F(h_>, h_\sim, h_>, h_<, h_<, h_<).
\end{split}
\end{equation*}

By the same argument as for $A$, we see that 
\begin{equation*}
\begin{split}
B_1&\lesssim s^{-1/6} \|h_>\|_2^2 \|h_\ll \|_2 \|h_<\|_2^3+\|h_>\|_2^2 \|h_\sim \|_2 \|h_<\|_2^3 \\
&=o_2(1) \|h_>\|_2^2.
\end{split}
\end{equation*}
Hence $B_1\lesssim o_2(1) \left\|e^{F(\cdot )} \widehat{f}_> \right \|_2^2$ where $o_2(1) \rightarrow 0$ as $s \rightarrow \infty$.  Following a similar process, we find that $B_k \lesssim \left\|e^{F(\cdot)}\widehat{f}_> \right\|_2^{k+1}$.  By setting $\mu=s^{-4}$ we have $e^{4\mu s^4}=e^4$.  Thus we have 
\begin{equation}
\begin{split}
\omega \left\|e^{F(\cdot)}\widehat{f}_>\right\|_2^2&\leq o_1(1) \left\|e^{F(\cdot)} \widehat{f}_>\right\|_2+o_2(1)\left\|e^{F(\cdot)} \widehat{f}_>\right\|_2^2+C\left\|e^{F(\cdot)} \widehat{f}_>\right\|_2^3+C\left\|e^{F(\cdot)} 
\widehat{f}_>\right\|_2^4 \\
&\qquad +\left\|e^{F(\cdot)} \widehat{f}_>\right\|_2^5+\left\|e^{F(\cdot)} \widehat{f}_>\right\|_2^6.
\end{split}
\end{equation}
Dividing both sides of inequality $(17)$ by $\left\|e^{F(\cdot)}\widehat{f}_>\right\|_2$ we obtain the desired result. 
\end{proof}

To see that $e^{\mu \xi^2}\widehat{f} \in L^2$ as required in Theorem \ref{thm-complex-analyticity}, we also follow a discussion in \cite{Shao}.  Define 
\begin{equation*}
\begin{split}
H(\varepsilon)&=\left(\int_{|\xi|\geq s^2} \left|e^{F_{s^{-4}, \varepsilon}(\xi)}\widehat{f} \right|^2 d\xi \right)^{1/2}.
\end{split}
\end{equation*}
Note here that $s$ is fixed, but we have control over that term.  By the dominated convergence theorem we have that $H(\varepsilon)$ is continuous on $(0, \infty)$, and is therefore connected on $(0, \infty)$.  To see that $H$ is bounded uniformly on $(0, \infty)$, consider the function $G(x)=\frac{\omega}{2}x-Cx^2-Cx^3-Cx^4-Cx^5$ on $(0, \infty)$ (refer to lemma $3.2$ and choose $s$ large enough such that $o(1)\leq \omega$).  This is very clearly bounded above by lemma $3.2$.  Let $M=\sup_{x \in [0, \infty)}G(x)$.  Notice that $G$ is concave, implying that the line $y=\frac{M}{2}$ intersects $G$ in at least two places, call the first two $x_0$ and $x_1$; clearly $x_0>0$.  Since $H(\varepsilon)$ is connected, then $G^{-1}\left(\left[0, \frac{M}{2}\right]\right)$ is contained in either $[0, x_0]$ or $[x_1,\infty)$.  When $s$ is sufficiently large and $\varepsilon=1$, $H(1)<x_0$.  Thus $G^{-1}\left(\left[0, \frac{M}{2}\right]\right)\subset[0, x_0]$.  This implies that $H(\varepsilon)$ is uniformly bounded on $(0, \infty)$.  Thus by Fatou's lemma or the monotone convergence theorem, $e^{\mu \xi^2}\widehat{f} \in L^2$ for $\mu=s^{-4}$. This finishes the proof of Theorem \ref{thm-complex-analyticity}.

\end{document}